\documentclass[11pt]{svmult}
\usepackage{amsmath}
\usepackage{mathptmx,helvet,courier,graphicx,makeidx,multicol,footmisc}
\usepackage{amssymb}
\usepackage{thmtools}
\usepackage{thm-restate}

\newcommand{\hms}[1]{\textup{LSP(#1)}}
\newcommand{\sparagraph}[1]{\noindent\textbf{#1}\quad}

\title*{High Multiplicity Scheduling with Switching Costs for few Products}
\author{Micha\"{e}l Gabay \and Alexander Grigoriev \and Vincent\,J.C.\,Kreuzen \and Tim Oosterwijk}
\institute{Micha\"{e}l Gabay\at{Laboratoire G-SCOP, Grenoble}, \email{michael.gabay@g-scop.grenoble-inp.fr} \and
Alexander Grigoriev \and Vincent\,J.C.\,Kreuzen \and Tim Oosterwijk\at{School of Business and Economics, Maastricht University}\\ \email{\{a.grigoriev,\,v.kreuzen,\,t.oosterwijk\}@maastrichtuniversity.nl}
}

\providecommand{\openbox}{\leavevmode
  \hbox to.77778em{%
  \hfil\vrule
  \vbox to.675em{\hrule width.6em\vfil\hrule}%
  \vrule\hfil}}
\makeatletter
\DeclareRobustCommand{\qed}{%
  \ifmmode
    \eqno \def\@badmath{$$}%$$
    \let\eqno\relax \let\leqno\relax \let\veqno\relax
    \hbox{\openbox}%
  \else
    \leavevmode\unskip\penalty9999 \hbox{}\nobreak\hfill
    \quad\hbox{\openbox}%
  \fi
}
\makeatother

\begin{document}
\date{}

\maketitle

\pagenumbering{arabic}

\abstract{
We study several variants of the single machine capacitated lot sizing problem with sequence-dependent setup costs and product-dependent inventory costs.
Here we are given one machine and $n \geq 1$ types of products that need to be scheduled. Each product is associated with a constant demand rate $d_i$, production rate $p_i$
and inventory costs per unit $h_i$. When the machine switches from producing product $i$ to product $j$, setup costs $s_{i,j}$ are incurred.
The goal is to minimize the total costs subject to the condition that all demands are satisfied and no backlogs are allowed.
\\
In this work, we show that by considering the high multiplicity setting and switching costs, even trivial cases of the corresponding ``normal'' counterparts
become non-trivial in terms of size and complexity. We present solutions for one and two products.
}

\section{Introduction\label{section:introduction}}

The area of High Multiplicity Scheduling is still largely unexplored.
Many problems that are easy in the normal scheduling setting become hard when lifted to their high multiplicity counterparts.
In this work, we study a single machine scheduling problem with sequence dependent setup costs called \emph{switching costs}, and under high multiplicity encoding of the input.
In this problem, we have a single machine which can produce different types of products.
Each day, only one type of product can be produced. Overnight, the machine can be adjusted to produce another type of product the following day.

\sparagraph{Related Work.}
High multiplicity scheduling problems have been investigated by several researchers.
Brauner et al.~\cite{BCGK05} provided a detailed framework for the complexity analysis of high multiplicity scheduling problems. We refer the reader to this paper for an excellent survey of related work in this field.
Madigan~\cite{madigan1968scheduling} and Goyal~\cite{goyal1973scheduling} both study a variant of our problem where setup costs only depend on the product which will be produced and holding costs are product-independent. The former proposes a heuristic for the problem, whereas the latter solves the problem to optimality for a fixed horizon.

\section{The Model and Basic Properties\label{section:preliminaries}}

We model the general problem for multiple products as follows:
we have a single machine that can produce a single type of product at any given time and we are given a set of products $J = \{1, \ldots, n\}$, and for each product $i\in J$, we are given a maximum production rate $p_i$, demand rate $d_i$ and holding costs $h_i$ per unit.
Furthermore, we are given switching costs $s_{i,j}$ for switching from producing product $i$ to producing product $j$.
The problem is to find an optimal cyclic schedule $S^*$ that minimizes the average costs per unit of time $\bar{c}(S^*)$.
Note that for each product $i$, the rates $d_i$ and $p_i$ and costs $h_i$ are assumed to be constant over time and positive.
Observe that the input is very compact. Let $m$ be the largest number in the input, then the input size is $\mathcal{O}(n\log m)$.
	
We distinguish three variants:
The \textbf{Continuous} case, where the machine can switch production at any time;
the \textbf{Discrete} case where the machine can switch production only at the end of a fixed unit of time e.g. a day;
and the \textbf{Fixed} case, where the machine can switch production only at the end of a fixed unit of time,
and each period in which the machine produces product $i$, a full amount of $p_i$ has to be produced (in the other cases, we can lower production rates).
We assume holding costs are paid at the end of each time unit.

We denote by \hms{A,$n$} with $A\in \{C,D,F\}, n\in \mathbb{N}$ the Lot-Sizing Problem of scheduling $n$ products in the Continuous, Discrete or Fixed setting.
Let $\pi_{i}^{[a,b]}$ denote the amount of product ${i}$ produced during time interval $[a,b]$. Let $\pi_{i}^{t}=\pi_{i}^{[t-1,t]}$.
Let $x_i^t$ be a binary variable denoting whether product $i$ is produced during time interval $[t-1,t]$.
Let $q_i^t$ denote the stock level for product $i$ at time $t$. We explicitly refer to the stock for a schedule $S$ as $q_i^t(S)$.

We now state some basic properties for the three variants.

	\begin{restatable}{lemma}{restateNPhard}\label{lemma:NPhard}
All three variants of the Lot Sizing Problem are strongly NP-hard.
	\end{restatable}
	\begin{proof}
		The lemma follows directly from a reduction from the Traveling Salesman Problem. \qed
	\end{proof}

\begin{restatable}{lemma}{restatefeasibility}\label{theorem:feasibility}
For all three variants of the problem, there exists a feasible schedule if and only if \;$\sum_{i\in J} {d_i}/{p_i} \leq 1$\;.
\end{restatable}
\begin{proof}
It is easy to see that ${d_i}/{p_i}$ is the fraction of time product $i$ needs to be scheduled on the machine and thus $\sum_{i \in J} {d_i}/{p_i}$ is at most 1. \qed
\end{proof}

\begin{restatable}{lemma}{restateidletimeCD}\label{prop:idletimeCD}
Let $S^*$ be an optimal schedule for \hms{C,$n$} or \hms{D,$n$}, with $n\in\mathbb{N}$. $S^*$ has no idle time.
\end{restatable}
\begin{proof}
If there is some idle time, we can simply decrease production rates to decrease holding costs. \qed
\end{proof}

\section{Single Product Case\label{section:k1}}

In most scheduling problems, scheduling a single product on a single machine is trivial.
However, considering a high multiplicity encoding takes away some of the triviality of this seemingly simple problem.

\sparagraph{Continuous Case\label{section:k1cont}.}
If a feasible schedule exists, we know that $p_1 \geq d_1$. In an optimal schedule, we produce to exactly meet demand, i.e. $\pi_1^{[a,b]} = d_1 (b-a)$.

\sparagraph{Discrete Case\label{section:k1disc}.}
If a feasible schedule exists, we know that $p_1 \geq d_1$. In an optimal schedule, we produce $d_1$ for every unit of time to exactly meet demand.

\sparagraph{Fixed Case\label{section:k1int}.}
The Fixed case for a single product is already non-trivial. We will prove the following theorem.

\begin{theorem}\label{theorem:optF1}
In an optimal schedule $S^*$ for \hms{F,1}, $\pi_1^t>0$ if and only if $q_1^{t-1} < d_1$.
\end{theorem}

We first characterize the minimum cycle length for \hms{F,1}, followed by the costs of an optimal schedule. The proof shows that for an optimal schedule $S^*$, the inventory levels for the time units in the schedule are the multiples of $\gcd(p_1,d_1)$ smaller than $p_1$.

\begin{lemma}\label{lemma:k1intlem1}
The minimum cycle length for \hms{F,1} is
\begin{equation}\label{eq:Int1Length}
l^* = \frac{p_1}{\gcd(p_1,d_1)}\;.
\end{equation}
\end{lemma}
\begin{proof}
Denote $\mathcal{G} = \gcd(p_1,d_1)$. Assume without loss of generality that $q_1^0<p_1$.
Since the cycle must be feasible, we have that $d_1 \leq p_1$.

Producing $p_1$ provides stock for $\left\lfloor {p_1}/{d_1}\right\rfloor $ time units, with a leftover stock of $p_1\mod d_1$.
Let stock at time $t$ be $q_1^t = q_1^{t-1}+\pi_1^t-d_1$. The schedule is cyclic when $q_1^t=q_1^0$ for $t>0$.
%\\
For a minimum cycle length, we want to minimize over $t$ such that $q_1^t = q_1^0 + \sum_{u=1}^t{\pi_1^u}-d_1 t = q_1^0$.
Rewriting gives \begin{equation*}\label{eq:rew}
t = \frac{\sum_{u=1}^t{\pi_1^u}}{d_1} = \sum_{u=1}^t{x_1^u} \frac{p_1}{d_1}.
\end{equation*}
Clearly, $t$ is minimized when $\sum_{u=1}^t{x_1^u} = \frac{d_1}{\mathcal{G}}$, and thus $t = \frac{p_1}{\mathcal{G}} = l^*$\;. \qed
\end{proof}

Using this lemma we compute the costs of an optimal schedule.

\begin{lemma}\label{lemma:optCycleInt1}
The shortest optimal cyclic schedule $S^*$ for \hms{F,1} has unit costs of
\begin{equation}\label{eq:Int1OPT}
\bar{c}(S^*) = \frac{h_1}{2} \left( p_1 - \gcd(p_1,d_1) \right)\;.
\end{equation}
\end{lemma}

\begin{proof}
Denote $\mathcal{G} = \gcd(p_1,d_1)$. Assume without loss of generality that the initial stock $q_1^0=0$ (see Remark~\ref{remark1} in the appendix).
Let $S^*$ be the optimal cyclic schedule with length $l^*$. Since $S^*$ is cyclic, $q_1^t$ has unique values for $t = 0, \ldots, l^*-1$.
Suppose $l^* > {p_1}/{\mathcal{G}}$.
Then each $q_1^t$ is a multiple of $\mathcal{G}$.
Since $l^* > {p_1}/{\mathcal{G}}$ and each $q_1^t$ has a unique value, there exists at least one $t$ such that $q_1^t\geq p_1$, and thus the schedule is not optimal.
Thus the length of the shortest optimal schedule is $l^* = {p_1}/{\mathcal{G}}$.

Since the total demand during the cycle is $d_1 l^*$ and each time unit of production produces $p_1$, we know that we produce during
${d_1 l^*}/{p_1} = {d_1}/{\mathcal{G}}$ time units. Since $q_1^t$ has a unique value for each $t < l^*$ and $q_1^0 = 0$, the stock values are all multiples of $\mathcal{G}$.
Hence, the values of $q_1^t$ are the multiples of $\mathcal{G}$ smaller than $p_1$. Since $p_1 = l^*\mathcal{G}$, the total stock for the cycle equals $\sum_{j=0}^{l^* - 1} j \mathcal{G}$.

Thus the total costs of $S^*$ are:
\begin{equation*}
h_1 \sum_{j=0}^{l^* - 1} j \mathcal{G} = h_1 \frac{1}{2} \mathcal{G} l^* (l^* - 1) = \frac{h_1 p_1}{2} \left( \frac{p_1}{\mathcal{G}} - 1 \right)\;. \qed
\end{equation*}
\end{proof}

The optimal schedule $S^*$ has length $l^*$ as in Eq.~(\ref{eq:Int1Length}), and total costs $l^*\bar{c}$ as in Eq.~(\ref{eq:Int1OPT}). The length of the cycle is linear in ${p_1}/{gcd(p_1,d_1)}$,
and Theorem~\ref{theorem:optF1} yields a polynomial delay list-generating algorithm.

\section{Continuous Case with two products\label{section:k2}}

Intuitively, the Continuous variant of the problem is less difficult than the Discrete one, which in turn is less difficult than the Fixed variant. In this section we show that for two products, even the Continuous case is already non-trivial.
We represent a cyclic schedule of length $C$ as a sequence:
\begin{equation*}\label{eq:sequence1}
 [t_0,t_1]_{j_0}^{r_0}, [t_1,t_2]_{j_1}^{r_1}, \ldots, [t_s,C]_{j_s}^{r_s} \;,
\end{equation*}
where $[t_i,t_{i+1}]_{j_i}^{r_i}$ denotes a \emph{phase} of the schedule, such that no two consecutive phases share the same $r_i$ and $j_i$, and in time interval $[t_i,t_{i+1}]$,
product $j_i\in J$ is produced at rate $r_i \leq p_{j_i}$.
A maximal sequence of consecutive phases of the same product $j_i$ is called a \emph{production period},
denoted by $[t_i,t_{i+1}]_{j_i}$.

We prove some structural results on the optimal schedule. The next lemma shows the machine produces every product $i$ only at rates $d_i$ and $p_i$ to minimize holding costs.

\begin{restatable}{lemma}{restatePhases}\label{lemma:k2contphases}
Consider \hms{C,$n$} for any $n \geq 2$. There is an optimal cycle $S^*$ such that for every product $i \in J$,
every production period of $i$ in $S^*$ consists of at most two phases.
For every production period, in the first phase the machine produces $i$ at a rate of $d_i$. During the second phase $i$ is produced at a rate of $p_i$.
\end{restatable}

We call a schedule a \emph{simple cycle} if there is exactly one production period for each product. The next lemma shows that in order to minimize holding costs, the optimal schedule for \hms{C,2} is a simple cycle.

\begin{lemma}\label{lemma:simplecycle1}
There exists an optimal schedule for \hms{C,2} that is a simple cycle.
\end{lemma}
\begin{proof}
Let $S^*$ be a minimal counterexample, i.e.
$S^* = [0,t_1]_1, [t_1,t_2]_2, [t_2,t_3]_1, [t_3,C]_2 $,
where $t_1 \neq (t_3-t_2)$. Now denote $A_1 = {(t_1+t_3-t_2)}/{2}$ and consider the following schedule,
\begin{equation*}
S = [0,A_1]_1, [A_1,C/2]_2, [C/2,C/2+A_1]_1, [C/2+A_1,C]_2\;,
\end{equation*}
which is obtained from $S^*$ by replacing the two production periods of each product by two production periods with averaged length.
Since $S^*$ is feasible, we have that $\pi_1^{[0,t_1]} + \pi_1^{[t_2,t_3]} \geq C d_1$ and $\pi_2^{[t_1,t_2]} + \pi_2^{[t_3,C]} \geq C d_2$.
Let $\pi_1^{[0,A_1]} = d_1 C/2$ in $S$ to cover the demand for product $1$ during the first two production periods. Let the production during the other production periods be similar.
Clearly, $S$ is feasible.
Note that $(t_2-t_1)+(C-t_3) = (C/2-A_1)+(C-C/2-A_1)$, i.e. the sum of the lengths of the production periods for product $i$ in $S$, is equal to that in $S^*$.

Now suppose there is in $S^*$ a production period $[a,b]$ for product $1$ with $q_1^a(S^*) > 0$.
Then during the production period $[x,a]_2$, holding costs increase by $q_1^a(S^*) h_1 (x-a)$ compared to $S$ and thus $\bar{c}(S) < \bar{c}(S^*)$.

Next, suppose $q_i^a(S^*)=0$ for every production period $[a,b]_i$.
It is easy to see that holding costs for product $1$ are only paid during production periods for $2$
and during the non-empty phase where product $1$ is produced at rate $p_1$. The same result holds for product 2.
Note that the sum of the lengths of the production periods for product $i$ in $S$, is equal to that in $S^*$ and holding costs are linear.
Hence, the area under the curve of the function of the holding costs over time, is the same in $S$ as in $S^*$, thus $\bar{c}(S) \leq \bar{c}(S^*)$.

Observe that $S$ consists of two simple cycles $S'$ and $S''$ with $S'=S''$. Thus $S'$ is a feasible simple cycle with the same unit costs as $S$. \qed
\end{proof}

For the rest of this section we assume without loss of generality that $h_1 < h_2$, and we only consider simple cycles. Next we show that an optimal schedule for \hms{C,2} consists of at most three phases.

\begin{lemma}\label{lemma:simplecycle2}
There exists an optimal schedule for any \hms{C,2} instance of the following form:
\begin{equation}\label{eq:simplecycle2}
S^* =  [0,t_1]_1^{p_1}, [t_1,t_2]_2^{d_2}, [t_2,C]_2^{p_2} \;,
\end{equation}
where the second phase is empty if and only if ${d_1}/{p_1} + {d_2}/{p_2} = 1$.
\end{lemma}
\begin{proof}
Let $S$ be an optimal cycle with four non-empty phases, i.e.
\begin{equation*}
S = [0,t_1]_1^{p_1}, [t_1,t_2]_2^{d_2}, [t_2,C]_2^{p_2}, [C,t_3]_1^{d_1}\;.
\end{equation*}
Consider the schedule consisting of only the first three phases, i.e. we remove $[C,t_3]_1^{d_1}$. Note that $\pi_2^{[t_2,C]} = d_2 \left( t_1 + ( t_3 - C ) \right) > d_2 t_1$. Hence the total amount of production for product $2$ can be lowered by $(t_3 - C)d_2$, by decreasing the length of phase $[t_2,C]_2^{p_2}$. Let $\alpha = (t_3 - C){d_2}/{p_2}$ and let
\begin{equation*}
S^* =  [0,t_1]_1^{p_1}, [t_1,t_2 + \alpha]_2^{d_2}, [t_2 + \alpha,C]_2^{p_2}\;.
\end{equation*}
Clearly $S^*$ is feasible and $\bar{c}(S^*) < \bar{c}(S)$.

If ${d_1}/{p_1} + {d_2}/{p_2} = 1$ the schedule is tight and demand can only be met by producing at maximum rate, which implies $[t_1,t_2 + \alpha]_2^{d_2}$ is empty.

If ${d_1}/{p_1} + {d_2}/{p_2} < 1$, there has to be a phase in which the machine does not produce at maximum rate, to avoid overproduction. By Lemma~\ref{lemma:k2contphases} there are at most two phases of production at rate $d_1$ and $d_2$ respectively. Since $h_1 < h_2$, by the above reasoning we introduce only one phase where we produce $d_2$ in order to minimize costs. \qed
\end{proof}

Using this result we calculate the optimal cycle length and corresponding costs. Let $S^*$ be as in Eq.~(\ref{eq:simplecycle2}). The costs of the schedule as a function of the parameter $t_1$, are given as
\begin{equation*}
\bar{c}(t_1) = \left( \frac{ h_1 (p_1 - d_1)}{2} + \frac{h_2 d_1 d_2}{2 p_1} \left( 1 + \frac{d_2}{p_2 - d_2} \right) \right) t_1 + \left((s_{1,2}+s_{2,1})\frac{d_1}{p_1} \right) \frac{1}{t_1}\;,
\end{equation*}
which is minimized for
\begin{equation*}
t^* = \sqrt{ \frac{ 2(s_{1,2}+s_{2,1})d_1 } { h_1 p_1 (p_1 - d_1) + h_2 d_1 d_2 \left( 1 + \frac{d_2}{p_2 - d_2} \right) } }\;.
\end{equation*}
The outcomes are summarized in the following theorem.
\begin{restatable}{theorem}{restateCalculations}\label{theorem:C2}
For \hms{C,2} there exists an optimal schedule of length $t^* {p_1}/{d_1}$ with average costs $\bar{c}(t^*)$.
\end{restatable}

\bibliographystyle{plain}
\bibliography{refs}

\begin{thebibliography}{1}

\bibitem{BCGK05}
Nadia Brauner, Yves Crama, Alexander Grigoriev, and Joris {V}an~{D}e Klundert.
\newblock {A framework for the complexity of high-multiplicity scheduling
  problems}.
\newblock {\em Journal of combinatorial optimization}, 9(3):313--323, 2005.

\bibitem{goyal1973scheduling}
SK~Goyal.
\newblock Scheduling a multi-product single machine system.
\newblock {\em Journal of the Operational Research Society}, 24(2):261--269,
  1973.

\bibitem{madigan1968scheduling}
JG~Madigan.
\newblock Scheduling a multi-product single machine system for an infinite
  planning period.
\newblock {\em Management Science}, 14(11):713--719, 1968.

\end{thebibliography}

\clearpage
\appendix

\renewcommand\thelemma{\thesection\arabic{lemma}}

\section{Basic Properties (Proofs for Lemma~\ref{lemma:NPhard} \& Lemma~\ref{theorem:feasibility})}

\restateNPhard*

\begin{proof}
We prove NP-hardness for the Discrete and Fixed cases by a reduction from the Traveling Salesman Problem. %~\cite{papadimitriou1977}.
Consider an instance $I=\{G=(V,E),c(i,j)_{i,j\in V}\}$ of the Traveling Salesman Problem. We construct an instance $I'=\{J,(d_i,p_i,h_i)_{i\in J},(s_{i,j})_{i,j\in J}\}$ of the Lot Sizing Problem as follows:
\\
Let $J=V$ and $s_{i,j}=c(i,j)$ for each $i,j\in J$. Let $d_i=1$, $p_i=|V|=n$ and $h_i=h= \sum_{j,k}{s_{j,k}}+1$ for each $i\in J$, and let $W_{\max} = \sum_{i,j}s_{i,j}$ and $W_{\min} = n \times \min_{i,j}{s_{i,j}}$.
The total cost for a schedule $S$ of length $C$ are defined as $c(S) = H + W$, with holding costs $\displaystyle H = \int_{0}^{C-1} {\sum_i{q_i^t h}}dt$. Note that for every feasible schedule, we have switching costs $W$ such that $W_{\min} \leq W$, and all simple cycles additionally satisfy $W \leq W_{\max}$.

Clearly, since demand and production rates are uniform, the stock level is constant over time. For every simple cycle there exists a feasible \emph{minimal} schedule of length $n$, using the same order of products, with average holding costs $\bar{H}=h n(n-1)/2$ and average switching costs $W_{\min}/n \leq \bar{W} \leq W_{\max}/n$. In fact, this schedule is minimum regarding the holding costs.

Let $S'$ be a feasible non-simple cycle of length $C'$ with total costs $c(S') = H'+W'$.
Note that at least two consecutive production periods of the same product are separated by $n+1$ time units. Hence, we need at least one additional unit in stock and thus $H' \geq h C' n(n-1)/2 + h C'$.
Thus, since $W\leq W_{\max} < h$, we have that the average costs of $S'$ are $\bar{c}(S') \geq H'/C' > \bar{c}(S)$ for every minimal simple cycle $S$.
Observe that the value of $\bar{H}$ is the same for every minimal simple cycle, and therefore the optimal solution to $I$ is the minimal simple cycle which minimizes $W$.

Let $\phi$ be a sequence of visits with costs $B$. Producing each product for 1 time unit with the same sequence as $\phi$ is a feasible solution for LSP with costs $hn(n-1)/2 + B/n$. Conversely, let $\phi$ be a solution for LSP with costs $hn(n-1)/2 + B/n$. This solution is a simple cycle, and therefore the production sequence is a tour with cost $B$. This proves the NP-hardness of the Discrete and the Fixed case.

\medskip

We prove the Continuous case by a similar reduction from the Metric TSP. We let $J = V$ and $s_{i,j} = c(i,j)$ for all $i,j \in J$. Let $d_i = 1$, $p_i = n$ and $h_i = 1$ for all $i \in J$.

Let $\phi$ be the optimal solution to $I$ with costs $c(\phi)$. Let $S$ be any feasible schedule for $I'$ of length $C$ with average costs $\bar{c}(S) = \bar{W} + \bar{H}$, where $\bar{W}$ are the average switching costs per time unit and $\bar{H}$ are the average holding costs per time unit. Let $S^*$ be the simple cycle of length $C^*$ where the products are produced in the same order as in $\phi$, with average costs $\bar{c}(S^*) = \bar{W}^* + \bar{H}^*$.

Since every product needs to be produced at least once in a feasible schedule and triangle inequality holds for the switching costs, we have that $W^* \leq W$. Note that in the Continuous setting, we can choose $C^*$ arbitrarily small. In particular, since holding costs decrease with the cycle length, we can choose $C^*$ such that $\bar{H}^* \leq \bar{H}$ and $\bar{c}(S^*) \leq \bar{c}(S)$. Thus we have that the optimal solution to $I'$ is a simple cycle $S^*$ using the sequence of $\phi$, which minimizes average costs.

Since all $p_i$ are equal and all $d_i$ are equal, every production period in the optimal schedule consists of one phase of length $C^* / n$ where the product is produced at rate $p_i=n$. Since $h_i = 1$, the total holding costs for every product $i$ are given as $$\int_0^{C^*/n} q_i^t dt + \int_{C^*/n}^C q_i^t dt = \frac{n-1}{2n}(C^*)^2\;,$$ and thus the total holding costs of $S^*$ are $H^*=(C^*)^2(n-1)/2$.

We know that the optimal solution $S^*$ to $I'$ minimizes the average costs, and thus the total holding costs are equal to the total switching costs. Hence we have
$$c(\phi)=W^*=H^*=\frac{n-1}{2}(C^*)^2\;,$$
which yields
$$C^*=\sqrt{\frac{2 c(\phi)}{n-1}}\;.$$

Now $\phi$ is an optimal solution for $I$ with costs $c(\phi)$, if and only if there is an optimal solution for $I'$ with average costs $\sqrt{2(n-1)c(\phi)}$. \qed
\end{proof}

\restatefeasibility*

\begin{proof}
Let $S$ be a feasible schedule of length $C$.
Then for each product $i$, the total demand during $S$ equals $C d_i$. Since we can produce at most $p_i$ during a time unit $t$, we know that
$$\frac{C d_i}{p_i} \leq \int_{C}{x_i^t dt}\;.$$
Summing over all products gives
$$\sum_{i\in J}{\frac{C d_i}{p_i}} \leq \sum_{i\in J}{\int_{C}{x_i^t dt}} \leq C\;.$$
Observe that since we can produce at most one type of product at any time, the right-hand side of the first inequality is at most $C$. Dividing by $C$ yields
$\sum_{i\in J}{{d_i}/{p_i}} \leq 1.$

Next, suppose that $\sum_{i\in J} {d_i}/{p_i} \leq 1$. Following the reverse of the proof above, we know that given some initial stock, we can now construct a feasible schedule.
Let $S$ be a schedule of length $C = \prod_{i\in J}{p_i}$.
Now, order the products in $J$ from $\{1,...,n\}$. For each product $i$, produce $\pi_i^{[t_{i-1},t_i]} = C d_i$, where $t_i = t_{i-1} + {C d_i}/{p_i}$ and $t_0=0$.
Clearly, given enough initial stock, demand is met for each product. Since $\sum_{i\in J}(t_i - t_{i-1}) = \sum_{i\in J}{{C d_i}/{p_i}} \leq C$,
all production fits in the cycle. Additionally, given integer demands and production rates, ${C d_i}/{p_i}$ is integer, ensuring feasibility for the Fixed case. \qed
\end{proof}

\section{Idle times (Proof for Lemma~\ref{prop:idletimeCD})}

\restateidletimeCD*
\begin{proof}
We prove by contradiction. Let $S$ be a counterexample,
i.e. there is at least one interval $[a,b]$ of length $(b-a)=t$ where the machine is idle.
Thus in this interval, each product $i$ has a demand $d_i t$ to fulfill. Therefore, for each $i$ there is a stock of at least $d_i t$ at time $a$, and thus for this interval, we pay at least $\sum_{i\in J}{d_i t h_i}$.
\\
Now, for each interval $[c,d]$ such that $c\geq b$ where only a single product $i$ is produced and the machine switches products at $d$, let $\pi_{i}^{[c,d]}\leftarrow\pi_{i}^{[c+t,d+t]}$.
Next, for each product $i$, choose interval $a^*,b^*$, such that $\pi_i^* := \pi_{i}^{[a^*,b^*]} > 0$ and $\pi_{i}^{[b^*,a]} = 0$.
Let $\pi_i^*\leftarrow \max\{\pi_i^*-d_{i}t,0\}$.
Clearly, the schedule is feasible, and we now pay at least $\sum_{i\in J}{\min\{\pi_i^*, d_i t\} h_i}$ less, and thus $S$ was not optimal. \qed
\end{proof}

\begin{remark}\label{remark1}
It is easy to see that we can assume that the initial stock is zero without loss of generality. Suppose that $q_1^0 \geq p_1$. Now reduce the stock by not producing for $t$ time units such that $q_1^t < p_1$.
Now suppose that $0 < q_1^0 < p_1$. Let $S$ be the optimal schedule and let $S^*$ be the optimal schedule given that $q_1^0=0$.
\\
If $q_1^0 = 0 \mod \mathcal{G}$, then since $S^*$ contains all multiples of $\mathcal{G}$ in the range $\{0,p_1 - 1\}$, $S$ is a permutation of $S^*$. The costs of the schedule stay the same.
\\
If $q_1^0 \neq 0 \mod \mathcal{G}$, the extra stock is obsolete. $S = S^*$, but each stock value will increase by $q_1^0 \bmod \mathcal{G}$. Thus the costs increase by $h_1 l^* (q_1^0 \bmod \mathcal{G})$, yielding total cycle costs of $\frac{h_1 p_1}{2} \left(\frac{p_1}{\mathcal{G}} - 1\right) + \frac{h_1 p_1}{\mathcal{G}} (q_1^0 \mod \mathcal{G})$.

\noindent Note that for multiple products, this assumption still holds for at least one product.
\end{remark}

\section{Continuous Case with two products (Proofs for Lemma~\ref{lemma:k2contphases} \& Theorem~\ref{theorem:C2})}
\restatePhases*
\begin{proof}

Let $P$ be a production period for product $i$ in an optimal cycle $S^*$.
Suppose the stock level $q_i$ decreases during $P$, i.e. $q_i^{t+\delta} = q_i^t - \varepsilon$ for some $t \in P, \delta > 0$.
Now let $\pi_i^{[0,t]} \leftarrow \pi_i^{[0,t]} - \epsilon$ and $\pi_i^{[t,t+\delta]} \leftarrow \pi_i^{[t,t+\delta]}+\varepsilon$. Clearly, the total holding costs are reduced and the schedule is still feasible. Thus in an optimal schedule, the production rate is never lower than $d_i$.

Next, suppose there is a phase $[a,c]_i^r$ s.t. $d_i < r < p_i$ with holding costs $h_i\frac{1}{2}(c-a)(q_i^c + q_i^a)$.
Now, let $[a,b]_i^{d_i}$ and $[b,c]_i^{p_i}$ where $b = \frac{c(r-p_i)+a(d_i-r)}{(d-p)}$, with $\int_0^b{q_i^t dt}=0$ and $q_i^c$ remains unchanged.
The holding costs are now $h_i\frac{1}{2}(c-b)(q_i^c + q_i^a)$ and thus $[a,c]_i^r$ was not optimal.
Thus we know that in an optimal schedule, each production period consists of consecutive phases of the form $[a,b]_i^{d_i},[b,c]_i^{p_i}$.

Now suppose that $P$ consists of more than two such phases. Thus there exists a time $t$ s.t. $[a,t]_i^{p_i},[t,b]_i^{d_i}$ with holding costs $h_i \left(\frac{1}{2}(t-a)(q_i^t + q_i^a) + (b-t)q_i^t \right)$.
Now swap the order of the two phases i.e. let $[a,t']_i^{d_i},[t',b]_i^{p_i}$ with $t'=a+(b-t)$, $q_i^b=q_i^t$ and holding costs $h_i \frac{1}{2}((b-a)+(b-t))(q_i^t + q_i^a)$.
%Calculation
Since holding costs decrease by $q_i^a t$, $P$ will consist of at most two phases of the form $[a,b]_i^{d_i},[b,c]_i^{p_i}$.
\end{proof}

\restateCalculations*
\begin{proof}
We denote the optimal schedule by $S^* =  [0,t]_1^{p_1}, [t,t']_2^{d_2}, [t',C]_2^{p_2} $.

We parametrize on $t$. Split the schedule in sub-schedules $S_1 = [0,t]_1^{p_1}$, $S_2 = [t,t']_2^{d_2}$ and $S_3 = [t',C]_2^{p_2}$. Let $c_i(S)$ denote the cost of (sub-)schedule $S$ for product $i$.
Note that $q_1^C=q_1^0=q_2^t=q_2^{t'}=0$, $q_2^C=q_2^0=d_2t$ and $q_1^t = t(p_1-d_1)$ and $q_1^{t'} = q_1^t-d_1(t'-t)$.
We calculate the average cost $\bar{c}(t)$ of $S^*$ as follows:
\begin{align*}
\bar{c}(t) & = \frac{1}{C} \left( c_1(S^*) + c_2(S^*) + s_{1,2} + s_{2,1} \right)\\
& =  \frac{1}{C} \left( c_1(S_1)+c_1(S_2+S_3) \right) + \frac{1}{C} \left( c_2(S_1)+c_2(S_3) \right) + \frac{s_{1,2} + s_{2,1}}{C} \\
& =  \frac{1}{C} \left( h_1\frac{1}{2}t^2(p_1-d_1) + h_1\frac{1}{2}t (p_1-d_1) (C-t) \right) + \frac{1}{C} \left( h_2\frac{1}{2}t^2d_2 + h_2\frac{1}{2} \frac{d_2 t}{p_2-d_2}d_2 t \right) + \frac{s_{1,2} + s_{2,1}}{C} \\
& =  \frac{h_1 t (p_1-d_1)}{2} + \frac{h_2t^2d_2}{2C} \left( 1 + \frac{d_2}{p_2-d_2} \right) + \frac{s_{1,2} + s_{2,1}}{C} \\
& = \left( \frac{ h_1 (p_1 - d_1)}{2} + \frac{h_2 d_1 d_2}{2 p_1} \left( 1 + \frac{d_2}{p_2 - d_2} \right) \right) t_1 + \left((s_{1,2}+s_{2,1})\frac{d_1}{p_1} \right) \frac{1}{t_1}\;.
\end{align*}
Since $c(t_1)$ is convex this expression is minimized when $\frac{dc(t_1)}{dt_1} = 0$. We find:

\begin{equation*}\label{eq:C2length}
t^* = \sqrt{ \frac{ 2(s_{1,2}+s_{2,1})d_1 } { h_1 p_1 (p_1 - d_1) + h_2 d_1 d_2 \left( 1 + \frac{d_2}{p_2 - d_2} \right) } } \;,
\end{equation*}

and thus the optimal average costs are equal to $\bar{c}(t^*)$. \qed
\end{proof}

\end{document}